\pdfoutput=1
\documentclass[11pt,reqno]{amsart}
\usepackage{amsmath,amssymb,amsthm,lmodern}
\usepackage[T1]{fontenc}
\usepackage[margin=1in]{geometry}
\usepackage{enumitem,booktabs}
\usepackage{float,needspace}
\usepackage{tikz}
\usetikzlibrary{calc,arrows.meta}
\definecolor{fanmark}{RGB}{25,91,116}
\tikzset{fanedge/.style={draw=black!65,line width=.55pt},
  markedwall/.style={draw=fanmark,line width=1.6pt},
  nonedge/.style={draw=black!75,line width=.5pt,dash pattern=on 1.2pt off 1.5pt},
  fanvertex/.style={circle,fill=white,inner sep=1.6pt,font=\small}}
\usepackage[round,authoryear]{natbib}
\usepackage{hyperref}
\hypersetup{hidelinks}
\theoremstyle{plain}
\newtheorem{theorem}{Theorem}[section]
\newtheorem{lemma}[theorem]{Lemma}
\newtheorem{proposition}[theorem]{Proposition}
\newtheorem{corollary}[theorem]{Corollary}
\theoremstyle{definition}
\newtheorem{definition}[theorem]{Definition}
\newtheorem{construction}[theorem]{Construction}
\theoremstyle{remark}

\DeclareMathOperator{\Span}{span}
\DeclareMathOperator{\mult}{mult}
\DeclareMathOperator{\Hom}{Hom}
\DeclareMathOperator{\prim}{prim}

\DeclareMathOperator{\PL}{PL}
\newcommand{\RR}{\mathbb R}
\newcommand{\ZZ}{\mathbb Z}
\newcommand{\QQ}{\mathbb Q}
\newcommand{\PP}{\mathbb P}
\newcommand{\Sig}{\Sigma}
\newcommand{\Gam}{\Gamma}
\newcommand{\cA}{\mathcal A}
\newcommand{\cC}{\mathcal C}
\newcommand{\cV}{\mathcal V}
\newcommand{\cJ}{\mathcal J}
\title[Sharp Bounds and Canonical Toric Projectivization]
{Sharp Bounds and Canonical Blow-ups for Toric Projectivization}
\author[P.~Bakhtary]{Parsa Bakhtary}
\address{Independent, Sunnyvale, California, USA}
\email{pbakhtary@gmail.com}
\urladdr{https://orcid.org/0000-0002-4491-5216}
\date{}
\begin{document}
\begin{abstract}
For every dimension $n\ge3$, we construct a family $\Sig_n(a,b)$ of
smooth complete fans with Picard number five and fixed combinatorial
type. For integers $a,b\ge1$, the minimum number of new rays in a smooth
projective refinement is $\min(a,b)$, attained by ordinary invariant
codimension-two blow-ups. If projective simplicial refinements with maximal-cone
multiplicity at most an integer $\mu\ge1$ are allowed, the minimum is
$\lceil\min(a,b)/\mu\rceil$. The threefold lower bound depends only on
subdivisions of four two-dimensional cones. Toric wedges preserve this
minimum while raising the dimension, and a selective support-function
criterion gives the matching upper bound. We also construct a
projectivization algorithm natural under lattice isomorphisms. For
smooth input, it principalizes an intrinsic monomial ideal by ordinary
blow-ups with smooth, possibly disconnected centers. Running a fixed
functorial principalization algorithm over $\QQ$ gives a sequence of
fans valid over every field. An intrinsic resolution procedure treats
arbitrary complete rational fans. Finally, Oda's threefold shows why
symmetry can prevent a canonical ordering of the components of a center,
even when their simultaneous blow-up is canonical.
\end{abstract}
\maketitle

\section{Introduction}\label{sec:intro}

Every complete toric variety admits a projective toric refinement. We
ask how many new rays such a refinement requires, and whether it can be
chosen naturally under lattice isomorphisms. These questions lead to two
different phenomena. Even in fixed dimension and at fixed Picard number,
a smooth projective refinement may require arbitrarily many new rays.
For smooth input, a natural construction is possible by blowing up
smooth invariant centers, provided that a center may have several
connected components.

\subsection{Conventions and notation}
Let $N\cong\ZZ^n$ be a lattice with dual $M=\Hom(N,\ZZ)$, and write
$N_\RR=N\otimes_\ZZ\RR$ and $M_\RR=M\otimes_\ZZ\RR$.
All cones are strongly convex rational
polyhedral cones: they are generated by finitely many lattice vectors
and contain no line. A fan is a finite collection of such cones, closed
under taking faces, in which two cones meet along a common face.
We call a $k$-dimensional cone a \emph{$k$-cone}. Write $|\Sig|$ for
the support, or union of the cones, and $\Sig(k)$ for the set of
$k$-cones. The fan is \emph{complete} if $|\Sig|=N_\RR$.
We write $\Gam\preceq\Sig$ when every cone of $\Gam$ lies in a cone
of $\Sig$ and $|\Gam|=|\Sig|$; such a fan is a \emph{refinement} of
$\Sig$. Its combinatorial type is the partially ordered set of its
cones under inclusion.

A ray is a one-dimensional cone. Its primitive generator is the first
nonzero lattice point on it; for $v\in N\setminus\{0\}$, we denote
the primitive generator of $\RR_{\ge0}v$ by $\prim(v)$.
A cone is \emph{simplicial} if its primitive ray generators are linearly
independent, and \emph{smooth} if they extend to a lattice basis. A fan
is simplicial or smooth when all its cones are so. For a simplicial cone
$\sigma=\langle u_1,\ldots,u_d\rangle$, set
\[
 \mult(\sigma)=[N\cap\Span_\RR\sigma:\textstyle\sum_i\ZZ u_i].
\]
This \emph{multiplicity} is one exactly when $\sigma$ is smooth; in
full dimension it is the absolute determinant of its generators.

Write $X_\Sig$ for the split toric variety associated to $\Sig$.
We call $\Sig$ \emph{projective} when $X_\Sig$ is projective.
A cone $\tau$ corresponds to an invariant orbit closure $V(\tau)$
of codimension $\dim\tau$. For a complete simplicial $n$-fan, its
\emph{Picard number} is
$\rho(X_\Sig)=\operatorname{rank}\operatorname{Pic}(X_\Sig)
=|\Sig(1)|-n$; we also use the term Picard rank.

\begin{definition}\label{def:cost}
For a complete simplicial $n$-fan $\Sig$ and an integer $\mu\ge1$, put
\[
 R_\mu(\Sig)=
 \min_{\substack{\Phi\preceq\Sig\;\mathrm{simplicial,\ projective}\\
                  \max_{\sigma\in\Phi(n)}\mult(\sigma)\le\mu}}
       \bigl(|\Phi(1)|-|\Sig(1)|\bigr).
\]
Thus $R_\mu$ is the fewest new rays needed when the output multiplicities
are at most $\mu$, and $R_1$ is the minimum for smooth projective
refinements. For smooth $\Sig$, let $\ell_{\min}(\Sig)$
be the minimum number of ordinary blow-ups at individual invariant orbit
closures of codimension at least two, with smooth intermediate varieties,
needed to reach a projective fan. Blowing up
$V(\langle u_1,\ldots,u_d\rangle)$ inserts the ray generated by
$u_1+\cdots+u_d$ and subdivides the cones containing that center cone.
\end{definition}

These minima are finite by classical toric projectivization
\citep[Theorem~4.5]{AMR}; for smooth input, finiteness of
$\ell_{\min}$ also follows directly from
Corollary~\ref{cor:adapt-projective} below. Each individual blow-up adds one
ray, so $R_1(\Sig)\le\ell_{\min}(\Sig)$ for smooth $\Sig$.

\subsection{Main results and context}
\begin{theorem}[Sharp bounds in every dimension]\label{thm:sharp}
For every $n\ge3$ and integers $a,b\ge1$, there is a smooth complete
fan $\Sig_n(a,b)$ with $n+5$ rays and Picard number five. For each fixed
$n$ its combinatorial type is independent of $a,b$, and, for every
integer $\mu\ge1$,
\[
 R_\mu(\Sig_n(a,b))=\left\lceil\frac{\min(a,b)}\mu\right\rceil,
 \qquad \ell_{\min}(\Sig_n(a,b))=\min(a,b).
\]
The smooth bound is attained by ordinary codimension-two blow-ups. In
particular,
\[
 \min_{\substack{\Phi\preceq\Sig_n(a,b)\\\Phi\;\mathrm{smooth,\ projective}}}
       \rho(X_\Phi)=5+\min(a,b).
\]
The ray lower bound applies to every refinement in the indicated class;
its morphism need not be given as a sequence of blow-ups.
\end{theorem}

The proof begins with the classical threefold family $Z'_{13}(a,b)$,
whose nonprojectivity was proved by \citet[\S4.1.2, Proposition~3]{Bonavero}
and \citet[Example~3]{FujinoPayne}. Projectivity imposes convexity
inequalities on four two-cones. An estimate on the rays at the ends
of their subdivisions turns these inequalities into the lower bound.
A chain inside one of the four cones attains it: inserting the sum of
the two current endpoint generators gives the smooth bound, while
advancing by up to $\mu$ lattice steps gives the bound under the
multiplicity cap $\mu$.
To reach higher dimensions, we use Ewald's construction, described in
Section~\ref{sec:higher}, which raises the dimension and the number of
rays by one \citep{Ewald,ChoiPark}. We prove that it preserves both
$R_\mu$ and $\ell_{\min}$. We also prove additivity under products for
the two smooth minima.

For the diagonal family $a=b=k$, the theorem excludes a bound depending
only on dimension and original combinatorial type, even at Picard number
five. For fixed $n$, binary ray coordinates and cone-index lists encode
the input in $\Theta(\log(k+1))$ bits, whereas a smooth projective output
has at least $k$ additional ray entries. Thus any algorithm listing all
output rays has output size exponential in the input size on this family.
This rules out polynomial running time for such an explicit output;
no hardness claim is made for compressed descriptions.

Our second result uses \emph{canonical} to mean natural under lattice
isomorphisms, without an auxiliary basis or input ordering; we use
\emph{intrinsic} in the same sense. Its main tool is principalization:
a sequence of blow-ups at smooth centers that makes the total transform
of an ideal invertible. We use a fixed principalization algorithm whose
choices are compatible with smooth morphisms in characteristic zero.

\Needspace{12\baselineskip}
\begin{theorem}[Canonical projectivization]\label{thm:canonical}
Fix the functorial principalization algorithm of
\citet{WResolution}. There is an algorithm assigning to every complete
rational fan $\Sig$ of dimension $n\ge2$ a smooth projective refinement
$\mathfrak P(\Sig)$ and a projective birational toric morphism to $X_\Sig$,
such that
\[
 \mathfrak P(T\Sig)=T\mathfrak P(\Sig)
\]
for every lattice isomorphism $T$. If $\Sig$ is smooth, the algorithm
gives a canonical sequence of ordinary blow-ups with smooth invariant
centers, possibly disconnected, and smooth intermediate fans. The fan
construction and these blow-up sequences are valid over every field.
\end{theorem}

Here a \emph{wall} is a cone of codimension one, and its \emph{wall
character} is a primitive element of $M$ vanishing on that cone,
determined up to sign. For smooth input, the algorithm principalizes
an intrinsic ideal built from the zero and pole divisors of these
characters. Each wall character then has one sign on each output cone:
it is everywhere nonnegative or everywhere nonpositive there.
We run the fixed algorithm over $\QQ$ and retain its fan sequence;
the resulting toric blow-ups and projectivity certificates are independent
of the field. This is an application of canonical principalization,
not a positive-characteristic resolution theorem. For nonsmooth input,
we first refine by the hyperplanes spanned by the walls and then make
the resulting fan smooth using intrinsic labels.
The centers in the smooth construction need not have codimension two.

The existence results are classical. Besides the toric Chow lemma,
one can use the hyperplanes spanned by the walls to obtain a projective
refinement \citep[Lemma~6.9.2]{Danilov}.
De Concini--Gaiffi construct subdivisions on which each character in a
prescribed finite set has one sign on every cone, using codimension-two
blow-ups of a smooth fan \citep[\S4]{DeConciniGaiffi}; see also monomial principalization
\citep{Goward,Harris}. Canonical equivariant toric desingularization is
available in a broader setting \citep[Addendum~1.2]{BierstoneMilman}.
Our contributions are the sharp ray bounds, their preservation under
Ewald's construction, a criterion using only selected wall characters,
and the intrinsic data used for projectivization.
We do not assert that the canonical toric construction is functorial under
arbitrary toric morphisms, commutes with restriction to open subfans, or
minimizes the number of new rays.

Section~\ref{sec:selective} establishes the support criterion and recalls
the adaptation algorithm. Sections~\ref{sec:threefold} and
\ref{sec:higher} prove Theorem~\ref{thm:sharp}.
Section~\ref{sec:canonical} proves Theorem~\ref{thm:canonical}, and
Section~\ref{sec:symmetry} explains why disconnected centers are essential.
Detailed certificate tables and secondary calculations are collected in
Online Resource~1. A notation index appears after Section~\ref{sec:verification}.

\section{Selective projectivity and sign-adaptation}\label{sec:selective}

\subsection{Support functions, subdivisions, and the wall arrangement}\label{subsec:arrangement}
A \emph{rational support function} on a fan is a continuous function
whose restriction to each cone is a linear form in $M_\QQ=M\otimes\QQ$.
We write $\PL_\QQ(\Sig)$ for the vector space of these functions.
On a complete fan we use the upper-convex convention: $h$ is
\emph{upper convex} if
$h(x+y)\ge h(x)+h(y)$, and \emph{strictly upper convex} if equality
holds exactly when $x,y$ lie in a common cone. This is the convexity
convention of \citet[Definition~6.1.4 and Theorem~6.1.14]{CLS}; our
strict upper convexity is their strict convexity, which characterizes
ample divisors.
A complete fan is projective if and only if it admits a rational
strictly upper convex support function. Clearing denominators gives
an ample Cartier divisor \citep[\S6]{CLS}. Equivalently, the fan is the
normal fan of a rational polytope; its cones record the faces on which
linear functionals attain their minima.

We will also use the fact that face multiplicities divide those of
containing simplicial cones. To see this, extend a basis of the saturated
span of the face to one of the containing cone's saturated span and
expand the determinant in this basis.

For a nonzero primitive $v\in|\Sig|\cap N$, the \emph{star subdivision}
$\Sig^*(v)$ consists of the cones not containing $v$ and the cones
$\langle v,F\rangle$, where $F$ is a face not containing $v$ of a cone
that contains $v$; this definition also applies to nonsimplicial fans
\citep[\S11.1]{CLS}.
For a simplicial cone containing $v$ in the relative interior of a face
$\tau$, the new maximal subcones adjoin $v$ and omit one generator of
$\tau$ at a time.
The morphism induced by a star subdivision is projective
\citep[Proposition~11.1.6]{CLS}.
On a smooth fan, subdivision at the sum of a cone's primitive generators
is the ordinary blow-up of its orbit closure. Subdivisions at cones
with no common containing cone commute; their orbit closures are disjoint.

Let $S_\Sig$ be the set of distinct \emph{unoriented} primitive wall
characters $[m]=\{m,-m\}$ of a complete fan $\Sig$. Since $m$ is a
primitive normal vector to the hyperplane spanned by its wall, we also
call it a \emph{wall normal}; the two terms are used interchangeably.
Cutting $N_\RR$ by
the hyperplanes $\ker m$ gives the \emph{wall arrangement fan}
$\cA_\Sig$: its cones are the closures of the regions cut out by these
hyperplanes and all their faces. The Minkowski sum of segments
$[-m,m]$ is a \emph{zonotope}. Write $|m|$ for the function
$x\mapsto|\langle m,x\rangle|$. Then
\begin{equation}\label{eq:zonotope}
 Z_\Sig=\sum_{[m]\in S_\Sig}[-m,m],\qquad
 h_{Z_\Sig}=-\sum_{[m]\in S_\Sig}|m|
\end{equation}
give a full-dimensional zonotope and its strictly upper convex support
function on $\cA_\Sig$ \citep[Theorem~7.16]{Ziegler}.
The arrangement refines $\Sig$: each maximal original cone is the
intersection of the half-spaces of its facets, so every arrangement
chamber meeting its interior is contained in it. The normals span
$M_\RR$, since the facets of a full-dimensional strongly convex cone
have trivial common annihilator. This also proves strong convexity of
the arrangement cones. Neither construction requires choices of signs
or an ordering, and nonsimplicial input is allowed.

\subsection{Obstruction transversals on arbitrary complete fans}\label{subsec:transversals}
Convexity can be tested one wall at a time. For
$h\in\PL_\QQ(\Sig)$ and adjacent maximal cones $\sigma_+,\sigma_-$
with common wall $\tau$, let $m_+,m_-$ be the linear pieces of $h$.
They agree on $\Span\tau$. Evaluating $m_+-m_-$ on a rational vector
of $\sigma_-\setminus\Span\tau$ defines a linear functional
$\delta_\tau(h)$, called the \emph{bend} at $\tau$. A different choice
of vector changes it by a positive rational factor. All bends are
positive exactly when $h$ is strictly upper convex.

For simplicial fans we use the primitive integral normalization
\begin{equation}\label{eq:bend}
 c_+r_++c_-r_-=\sum_i c_i r_i,\qquad
 \delta_\tau(h)=\sum_i c_i h(r_i)-c_+h(r_+)-c_-h(r_-),
\end{equation}
where $\tau=\langle r_1,\ldots,r_{n-1}\rangle$, the outside generators
are $r_+,r_-$, $c_+,c_->0$, and the gcd of \emph{all} coefficients is one.
For smooth cones $c_+=c_-=1$; this need not hold for singular cones.

\begin{definition}\label{def:transversal}
The \emph{obstruction cone} of $\Sig$ is
\[
 \cC(\Sig)=\{q\in\RR_{\ge0}^{\Sig(n-1)}:
                   \textstyle\sum_\omega q_\omega\delta_\omega=0
                         \text{ on }\PL_\QQ(\Sig)\}.
\]
A nonzero $q\in\cC(\Sig)$ is a nonnegative dependence among the bends;
its \emph{support} is $\{\omega:q_\omega>0\}$.
A wall set $\cV$ is an \emph{obstruction transversal} if it meets the
support of every such $q$. Let $S_\cV$ be the set of distinct unoriented
primitive normals of these walls. A fan is \emph{$m$-adapted} if each
cone lies in one of the two closed half-spaces bounded by $\ker m$.
Making a fan $m$-adapted by subdivision is called \emph{sign-adaptation}.
\end{definition}

The obstruction is immediate: a nonzero nonnegative combination of
positive bends cannot vanish. A transversal therefore meets every
dependence that prevents all bends from being positive.
Gordan's alternative shows that $\cV$ is a transversal exactly when some
$h_0\in\PL_\QQ(\Sig)$ satisfies $\delta_\omega(h_0)>0$ outside $\cV$.
Indeed, apply the alternative to the remaining bend rows in any rational
basis of $\PL_\QQ(\Sig)$; a feasible strict rational system has a rational
solution. On a nonsimplicial fan this space includes the cone-linearity
constraints: arbitrary ray values are not sufficient.

For $\Gam\preceq\Sig$, a function $g\in\PL_\QQ(\Gam)$ is
\emph{strictly upper convex relative to $\Sig$} if its bend is positive
at every wall of $\Gam$ whose relative interior lies in the interior of
a maximal cone of $\Sig$; such functions correspond to relatively ample
rational Cartier divisors for $X_\Gam\to X_\Sig$ \citep[Theorems~7.2.11--7.2.12]{CLS}.

\Needspace{12\baselineskip}
\begin{theorem}[Selective support criterion]\label{thm:selective}
Let $\Sig$ be any complete rational fan and $\cV$ an obstruction
transversal. Suppose $\Gam\preceq\Sig$ is adapted to every normal in
$S_\cV$ and $X_\Gam\to X_\Sig$ is projective. Then $\Gam$ is projective.
If $h_0$ has positive bends outside $\cV$ and $g\in\PL_\QQ(\Gam)$ is
strictly upper convex relative to $\Sig$, then
\begin{equation}\label{eq:certificate}
 H=h_0-C\sum_{[m]\in S_\cV}|m|+\varepsilon g
\end{equation}
is strictly upper convex for every sufficiently large $C>0$ and then
every sufficiently small $\varepsilon>0$ (depending on $C$).
\end{theorem}
\begin{proof}
The three terms in~\eqref{eq:certificate} serve separate purposes:
$h_0$ handles the unselected original walls, the absolute-value terms
handle the selected walls, and $g$ handles walls created inside original
maximal cones.

A wall of $\Gam$ is either interior to an original maximal cone or
inherited from an original wall. The pullback of $h_0$ has zero bend
on an interior wall and a positive multiple of its original bend on
an inherited wall: the difference of the adjacent linear pieces is the
same multiple of the original normal. This observation does not require
simpliciality.

Put $\Psi=-\sum_{[m]\in S_\cV}|m|$. Adaptation makes $\Psi$ linear
on every cone. Its bends are nonnegative, and the term $-|m|$ has
strictly positive bend exactly on walls contained in $\ker m$.
Thus $h_0+C\Psi$ has positive bends on unselected inherited walls for
every $C\ge0$, and on selected inherited walls for sufficiently large
$C$. Choose such a $C$. There are finitely many inherited walls, so
adding a sufficiently small positive multiple of $g$ preserves their
positive bends. At each interior wall the bend of $h_0$ is zero, that
of $\Psi$ is nonnegative, and that of $g$ is positive. All bends of
$H$ are therefore positive. The choices can be rational.
\end{proof}

Selecting all walls always works, with $h_0=0$ and
$H=h_{Z_\Sig}+\varepsilon g$. Selection can be made by normal sets:
write $[m_\omega]$ for the normal of a wall $\omega$ and, for
$S\subseteq S_\Sig$, set $\cV(S)=\{\omega:[m_\omega]\in S\}$.
By the preceding alternative, $\cV(S)$ is a transversal exactly when
$\delta_\omega(h_0)\ge1$ for $\omega\notin\cV(S)$ is feasible. Thus there is no need
to enumerate the rays of $\cC(\Sig)$.

\subsection{The elementary adaptation algorithm}
Fix an ordered lattice basis for the following deterministic version
of sign-adaptation. Orient each selected primitive normal so that its
first nonzero coordinate is positive, and order the normals
lexicographically. Break ties by the lexicographically smallest
pair of primitive ray vectors, with each pair itself sorted
lexicographically in the chosen basis.
This construction is \emph{basis-dependent}: its deterministic choices
use the specified basis. Section~\ref{sec:canonical} gives an
intrinsic construction.

For a character $m$, call a two-cone $\langle u,v\rangle$
\emph{bad} if $m(u)m(v)<0$, and call $|m(u)|+|m(v)|$ its weight.

\begin{proposition}[Sign-adaptation and multiplicities]\label{prop:adapt}
For a finite simplicial fan and a primitive character $m$, repeatedly
choose a bad two-cone $\langle u,v\rangle$ of maximum weight, and
insert $s=\prim(u+v)=(u+v)/d$.
The procedure terminates with an $m$-adapted fan. For each affected
maximal cone $\sigma$, both children have multiplicity
$\mult(\sigma)/d$. Smooth input therefore gives ordinary
codimension-two blow-ups and smooth intermediate fans. Sequential
adaptation to finitely many characters preserves earlier adaptations.
\end{proposition}
\begin{proof}
The key point is that every new bad two-cone has smaller weight than
the selected one. Write $m(u)=p>0$, $m(v)=-q<0$, and $W=p+q$.
If $p>q$, then $0<m(s)=(p-q)/d\le p-q$. The only possible bad
replacement of the selected edge has weight at most $p<W$.
Every other new edge is $\langle s,w\rangle$ for a ray $w$ in an
affected cone. If it is bad, then $m(w)<0$, and the old edge
$\langle u,w\rangle$ was bad. Hence its new weight is at most
\[
 (p-q)+|m(w)|<p+|m(w)|\le W.
\]
The case $q>p$ is symmetric; if $p=q$ then $m(s)=0$ and no new bad
edge occurs. Let $W_0$ be the initial maximum weight and $c_j$ the number
of bad two-cones of weight $j$. The tuple $(c_{W_0},\ldots,c_1)$ strictly
decreases lexicographically at each step, and no larger weight appears.
Such tuples admit no infinite lexicographically decreasing sequence,
so the procedure terminates.
On a simplicial cone every pair of rays
spans a two-face; absence of bad edges is therefore adaptation.

Multilinearity of the determinant in the saturated span gives both
child multiplicities $\mult(\sigma)/d$. In the smooth case $d=1$ and
the insertion is an ordinary blow-up. Finally, any positive combination
of generators in one closed half-space stays in that half-space, so
later subdivisions preserve adaptation to an earlier character.
\end{proof}

These subdivision moves and the existence of smooth sign-adaptation
are those of \citet[\S4]{DeConciniGaiffi}. The maximal sum-weight rule
specifies the version used here. Its weight drop is also the elementary
identity
\[
 |m(u)|+|m(v)|-|m(u+v)|=2\min(|m(u)|,|m(v)|)
\]
on a bad two-cone. The termination argument uses the weight counts
for one character at a time.

\begin{corollary}\label{cor:adapt-projective}
For a complete simplicial fan and an obstruction transversal,
sequential adaptation to its normals gives a projective simplicial
refinement without increasing maximal-cone multiplicities. For smooth
input all steps are ordinary codimension-two blow-ups.
\end{corollary}
\begin{proof}
The composite of star subdivisions is projective, so it admits a
rational function $g$ strictly upper convex relative to the input fan. Apply
Theorem~\ref{thm:selective}.
\end{proof}

Taking all wall normals and using the ordered-basis conventions above
recovers the basis-canonical projectivization theorem of the
\href{https://arxiv.org/abs/2605.25013v3}{earlier version}, with every step
an ordinary blow-up along a smooth invariant codimension-two center.

For nonsimplicial input the support criterion remains valid, but the
absence of bad two-faces need not imply adaptation. One may first
take a projective simplicial subdivision and then adapt to the selected
characters. Also, minimizing the number of selected normals need not
minimize the subdivision count, as the next section shows.

\section{The sharp threefold tradeoff}\label{sec:threefold}

\subsection{The family and an endpoint estimate}
We now compute both minima for a family of three-dimensional fans.
The lower bound will use only four two-cones, while the upper bound
will come from subdividing a single one.
For integers $a,b\ge1$, let $\Sig(a,b)=Z'_{13}(a,b)$ have rays
\begin{align*}
 v_1&=(-1,b,0),&v_2&=(0,-1,0),&v_3&=(1,-1,0),&v_4&=(-1,0,-1),\\
 v_5&=(0,0,-1),&v_6&=(0,1,0),&v_7&=(0,0,1),&v_8&=(1,0,a),
\end{align*}
and maximal cones
\begin{equation}\label{eq:family-cones}
 124,\ 234,\ 345,\ 456,\ 146,\ 167,\ 127,\ 237,\ 358,\ 568,\ 378,\ 678.
\end{equation}
Here $ij$ and $ijk$ abbreviate $\langle v_i,v_j\rangle$ and
$\langle v_i,v_j,v_k\rangle$, respectively. This is the smooth
complete fan labelled $[8\text{-}13']$ in \citet[\S9]{Oda1978}, with
the coordinates of \citet[\S4]{FujinoSato}. It has eight rays, eighteen
two-cones, and twelve maximal cones; we record these counts as its
\emph{$f$-vector}, $(8,18,12)$. Its Picard number is five.
Bonavero's parameters $(a',b')$ correspond to $(a,b-1)$ here
\citep[\S4.1.2, Proposition~3]{Bonavero}.
\citet[Example~3]{FujinoPayne} prove the stronger statement that
$X_{\Sig(a,b)}$ has no nontrivial nef line bundle for all $a,b\ge1$
except $(a,b)=(1,1)$. This exception is still nonprojective:
Theorem~\ref{thm:threefold} gives $R_1(\Sig(1,1))=1$.

Figure~\ref{fig:family} displays its triangulated sphere, the four
non-adjacent pairs used in the lower bound, and the chain used for the
smooth upper bound.
\begin{figure}[htbp]
\centering
\begin{minipage}[c]{.49\textwidth}
\centering
\begin{tikzpicture}[x=.82cm,y=.82cm]
\coordinate (f1) at (-3,-2);
\coordinate (f6) at (3,-2);
\coordinate (f7) at (0,3);
\coordinate (f2) at (-.24,.78);
\coordinate (f3) at (.36,-.02);
\coordinate (f4) at (-1.44,-1.14);
\coordinate (f5) at (.98,-1.10);
\coordinate (f8) at (.78,.71);
\fill[black!3] (f1)--(f6)--(f7)--cycle;
\foreach \a/\b in {1/2,1/4,1/6,1/7,2/3,2/4,2/7,3/4,3/5,3/7,3/8,4/5,4/6,5/6,5/8,6/7,6/8,7/8}
  \draw[fanedge] (f\a)--(f\b);
\foreach \a/\b in {4/7,1/5,3/6,2/8}
  \draw[nonedge] (f\a)--(f\b);
\foreach \a/\b in {1/2,3/7,4/6,5/8}
  \draw[markedwall] (f\a)--(f\b);
\foreach \i in {1,...,8}
  \node[fanvertex] at (f\i) {$\i$};
\node[font=\small] at (0,-2.55) {(a) Sphere of $\Sig(a,b)$};
\end{tikzpicture}
\end{minipage}\hfill
\begin{minipage}[c]{.49\textwidth}
\centering
\begin{tikzpicture}[x=.9cm,y=.9cm]
\coordinate (a) at (-2.45,0);
\coordinate (b) at (2.45,0);
\coordinate (c) at (0,1.9);
\coordinate (d) at (0,-1.9);
\coordinate (s) at (-1.35,0);
\coordinate (t) at (1.10,0);
\draw[fanedge] (a)--(c)--(b)--(d)--cycle;
\draw[markedwall] (a)--(b);
\foreach \q in {s,t} {
  \draw[fanedge,dashed] (c)--(\q)--(d);
  \fill[fanmark] (\q) circle (1.6pt);
}
\node[fill=white,inner sep=2pt] at (0,0) {$\cdots$};
\node[fanvertex] at (a) {$1$};
\node[fanvertex] at (b) {$2$};
\node[fanvertex] at (c) {$4$};
\node[fanvertex] at (d) {$7$};
\node[below=3pt,font=\small,fill=white,inner sep=1pt] at (s) {$s_1$};
\node[below=3pt,font=\small,fill=white,inner sep=1pt] at (t) {$s_b$};
\node[font=\small] at (0,-2.45) {(b) Chain in $12$};
\end{tikzpicture}
\end{minipage}
\caption{Combinatorial diagrams, independent of $a,b$.
Left: the thick edges are $12,37,46,58$; the unbounded face represents
$167$, completing the sphere. The dotted chords $47,15,36,28$ join the
four pairs whose sums $P_0,P_1,Q_0,Q_1$ enter~\eqref{eq:four-pairs};
none is an edge, so each pair lies in no common cone. Each chord
crosses the thick edge that contains the corresponding sum.
Right: the smooth chain
$s_j=v_1+jv_2$ subdivides the two adjacent triangles $124,127$;
dashed edges join each inserted ray to $4$ and $7$. Positions are schematic.}
\label{fig:family}
\end{figure}

\Needspace{13\baselineskip}
The following estimate bounds how far a subdivision can move the
generator next to one boundary ray when the number of inserted rays
and their cone multiplicities are prescribed.

\begin{lemma}[Endpoint slope with bounded multiplicity]\label{lem:endpoint}
Subdivide a smooth two-cone $\langle u,v\rangle$ with $r$ interior
rays, all resulting two-cones having multiplicity at most $\mu$.
Write the primitive generator adjacent to $v$ as $cu+tv$. Then
$c>0$, $t\ge0$, and $t/c\le\mu r$.
\end{lemma}
\begin{proof}
If $r=0$, the adjacent generator is $u$, and the assertion is immediate.
Otherwise list the interior generators as $(c_i,t_i)$ in the lattice
basis $(u,v)$, in order from $u$ to $v$, and
put $(c_0,t_0)=(1,0)$. All $c_i$ are positive integers. The determinant
$d_i=c_{i-1}t_i-t_{i-1}c_i$ is the multiplicity of the intervening cone,
so $1\le d_i\le\mu$. Thus
\[
 \frac{t_r}{c_r}=\sum_{i=1}^r\frac{d_i}{c_{i-1}c_i}
                    \le\mu r.\qedhere
\]
\end{proof}

\begin{theorem}\label{thm:threefold}
For every $a,b,\mu\ge1$,
\[
 R_\mu(\Sig(a,b))=\left\lceil\frac{\min(a,b)}\mu\right\rceil,
 \qquad \ell_{\min}(\Sig(a,b))=\min(a,b).
\]
The first minimum is attained by star subdivisions at two-cones with
all intermediate multiplicities at most $\mu$. For $\mu=1$ these are
ordinary smooth blow-ups.
\end{theorem}

\begin{proof}[Proof of the lower bound]
Let $\Phi\preceq\Sig(a,b)$ be projective and simplicial with maximal
multiplicity at most $\mu$, and let $h$ be strictly upper convex.
Let $r_i$ count the new rays in the relative interior of $\tau_i$, where
\[
 \tau_1=12,\quad\tau_2=37,\quad\tau_3=46,\quad\tau_4=58.
\]
These interiors are disjoint, so $\Phi$ has at least $\sum_i r_i$ new
rays. It is therefore enough to prove
\begin{equation}\label{eq:disjunction}
 \mu(r_1+r_3)\ge b\qquad\text{or}\qquad\mu(r_2+r_4)\ge a.
\end{equation}
We argue by contradiction, comparing two estimates for the support
function at the following four vectors:
\begin{align*}
 P_0&=v_4+v_7=v_1+bv_2,&P_1&=v_1+v_5=v_4+bv_6,\\
 Q_0&=v_3+v_6=v_8+av_5,&Q_1&=v_2+v_8=v_3+av_7
\end{align*}
In each first expression, the two summands lie in no common cone of
the original fan (the dotted chords in Figure~\ref{fig:family}), hence
in no common cone of $\Phi$. Strict upper
convexity gives
\begin{equation}\label{eq:four-pairs}
 h(P_0)+h(P_1)+h(Q_0)+h(Q_1)>\sum_{i=1}^8h(v_i).
\end{equation}
Also $h(v_2)+h(v_6)<0$ and $h(v_5)+h(v_7)<0$: the pairs are antipodal,
so strict upper convexity bounds each sum by $h(0)=0$.

Suppose both inequalities in~\eqref{eq:disjunction} fail. Their integer
left sides are then at most $b-1$ and $a-1$, respectively.
We will show that too few rays force the reverse
of~\eqref{eq:four-pairs}. The restrictions of $\Phi$ to the four original
two-cones have multiplicities at most $\mu$, since face multiplicities
divide those of maximal cones.
In $12,46,58,37$, take the generator adjacent to $v_2,v_6,v_5,v_7$,
respectively, and divide it by its coefficient of $v_1,v_4,v_8,v_3$.
Lemma~\ref{lem:endpoint}, with $v$ equal to the corresponding adjacent
ray, gives the real vectors
\begin{gather*}
 u=(-1,\alpha,0),\quad b-\mu r_1\le\alpha\le b,
 \qquad w=(-1,\beta,-1),\quad0\le\beta\le\mu r_3,\\
 p=(1,0,\gamma),\quad a-\mu r_4\le\gamma\le a,
 \qquad q=(1,-1,\eta),\quad0\le\eta\le\mu r_2.
\end{gather*}
Thus $\alpha-\beta\ge b-\mu(r_1+r_3)\ge1$ and, similarly,
$\gamma-\eta\ge1$. In particular, $\alpha\ge1$ and $b-\beta\ge1$
(and likewise $\gamma\ge1$ and $a-\eta\ge1$).
The endpoint cones are $\langle u,v_2\rangle$ and
$\langle w,v_6\rangle$, so linearity applies to
$P_0=u+\alpha v_2$ and $P_1=w+(b-\beta)v_6$.
Superadditivity at $v_1=u+(b-\alpha)v_6$ and $v_4=w+\beta v_2$
then gives
\begin{align*}
 h(P_0)+h(P_1)
 &\le h(v_1)+h(v_4)+(\alpha-\beta)(h(v_2)+h(v_6))\\
 &\le h(v_1)+h(v_4)+h(v_2)+h(v_6).
\end{align*}
Here the last inequality uses $h(v_2)+h(v_6)<0$ and
$\alpha-\beta\ge1$. The same computation with $Q_0=p+\gamma v_5$,
$Q_1=q+(a-\eta)v_7$, $v_8=p+(a-\gamma)v_7$, and
$v_3=q+\eta v_5$ bounds the other sum by
$h(v_8)+h(v_3)+h(v_5)+h(v_7)$, contradicting~\eqref{eq:four-pairs}.
The normalized vectors need not be integral; positive homogeneity is
all that is used. This proves~\eqref{eq:disjunction}, hence
$\sum_i r_i\ge\lceil\min(a,b)/\mu\rceil$.
For $\mu=1$, each individual ordinary blow-up adds one ray, so the
same lower bound holds for the length, including steps of higher codimension.
\end{proof}

\subsection{Selective certificates and attainment}
To attain the bound, we need only adapt to a single character.
The wall set $\{58\}$, with normal $m_y=(0,1,0)$, is an obstruction
transversal. Although $58$ is the selected wall, adaptation to $m_y$
starts by subdividing the bad two-cone $12$: the transversal selects
which character to use, and adaptation subdivides wherever that
character changes sign. To certify the transversal,
assign the original rays the values
\begin{equation}\label{eq:singleton}
 \widehat h=(-(4b+1),0,0,2,0,-4,-4(b+1),-4a(b+1)-2).
\end{equation}
Its eighteen bends, computed from~\eqref{eq:bend}, are listed in
Table~\ref{tab:singleton-bends};
all but $\delta_{58}$ are positive.
\begin{table}[H]
\centering
\caption{Primitive wall bends of the singleton-transversal certificate $\widehat h$}
\label{tab:singleton-bends}
\begin{tabular}{@{}lc@{\qquad}lc@{}}
\toprule
Wall(s)&Bend&Wall(s)&Bend\\
\midrule
$12$&$1$&$14,17$&$4$\\
$16$&$4b+1$&$23$&$4b+2$\\
$24,27$&$4b+1$&$34,37$&$2$\\
$35,56$&$4a(b+1)$&$38,68$&$4(b+1)$\\
$45,78$&$2$&$46,67$&$3$\\
$58$&$2-4a(b+1)$&&\\
\bottomrule
\end{tabular}
\end{table}

\begin{proof}[Proof of attainment in Theorem~\ref{thm:threefold}]
The unique $m_y$-bad two-cone initially is $12$. Set
\[
 t_0=0,\qquad t_j=\min(j\mu,b),\quad
 s_j=v_1+t_jv_2=(-1,b-t_j,0),\quad
 1\le j\le L:=\lceil b/\mu\rceil.
\]
With $s_0=v_1$, insert
$s_j=s_{j-1}+(t_j-t_{j-1})v_2$ in $\langle s_{j-1},v_2\rangle$.
All these rays are primitive, since their first coordinate is $-1$.
The other rays of affected maximal cones
are $v_4,v_7$, with zero $y$-coordinate. The final fan is therefore
$m_y$-adapted. Its chain cones have multiplicities $t_j-t_{j-1}\le\mu$;
the terminal cones adjacent to $v_2$ have multiplicity one. At every
intermediate stage the remaining endpoint cone is smooth, and all
other cones are unchanged or already have multiplicity at most $\mu$.
The composite of star subdivisions is projective, so the singleton
certificate and Theorem~\ref{thm:selective} make the output projective.

The lattice involution $T(x,y,z)=(-x,z,y)$ sends $\Sig(a,b)$ to
$\Sig(b,a)$ and permutes ray indices by $(1\ 8)(2\ 5)(3\ 4)(6\ 7)$.
It exchanges $y$- and $z$-adaptation and preserves the cone list.
Thus $\lceil a/\mu\rceil$ rays also suffice. Taking the shorter chain
gives the asserted minimum. For $\mu=1$ each insertion is the sum of
the current endpoint generators, giving ordinary codimension-two
blow-ups in exactly $b$ or $a$ steps.
\end{proof}

The argument also gives a bound under a weaker smoothness hypothesis.
A toric variety is \emph{smooth in codimension two} when its singular
locus has codimension at least three, or equivalently when every cone
of dimension at most two is smooth.

\begin{corollary}[Local regularity suffices]\label{cor:codim-two}
Every projective refinement of $\Sig(a,b)$ that is smooth in codimension
two has at least $\min(a,b)$ new rays, even if it is nonsimplicial.
In particular, allowing only isolated singularities in the output does
not lower the minimum ray count.
\end{corollary}
\begin{proof}
The lower-bound proof only uses smoothness of the subdivisions on
$12,37,46,58$. All two-cones are simplicial, and smoothness in
codimension two makes their multiplicities one. The support inequalities
do not require simplicial maximal cones.
\end{proof}

For the diagonal family $a=b=k$, one insertion at $(-1,0,0)$ gives a
projective refinement of maximal
multiplicity $k$, whereas a smooth output requires $k$ rays. More
generally, $r$ new rays and maximal multiplicity $\mu$ must satisfy
$r\mu\ge\min(a,b)$, and Theorem~\ref{thm:threefold} realizes the exact
tradeoff for every cap. Choosing only one normal need not minimize the
number of blow-ups: on $\Sig(1,b)$ the two coordinate choices cost
$1$ and $b$ ordinary steps. Full adaptation with the standard basis
and lexicographic normal order takes $42$ steps on $\Sig(2,2)$, compared
with two for either coordinate choice; the exact run is in Online Resource~1.

\section{Wedges, higher dimensions, and products}\label{sec:higher}

\subsection{Preservation under a toric wedge}\label{subsec:wedge}
To preserve Picard number while raising dimension, we must add exactly
one ray at each step. The following construction does this and preserves
both projectivization minima. It is the canonical extension of
\citet{Ewald}, also called the \emph{toric wedge}; see
\citet[\S\S4--5]{ChoiPark}. Let $\Sig$ be complete and
simplicial, and distinguish a primitive ray $v$. In $N\oplus\ZZ$ replace
$v$ by $v_+=(v,1)$ and $v_-=(0,-1)$, and lift every other generator
$u$ to $(u,0)$. For a maximal cone containing $v$, replace that ray by
both $v_+,v_-$. For a maximal cone not containing $v$, take its two
extensions by $v_+$ and by $v_-$. Denote this fan by $W_v\Sig$.
The \emph{quotient star} at a ray $\rho$ is the fan in $N/\ZZ u_\rho$
formed by the images of cones containing $\rho$, where $u_\rho$ is its
primitive generator. It is the fan of the
orbit closure $V(\rho)$ \citep[Proposition~3.2.7]{CLS}.

For example, wedging the fan of $\PP^1$ at $v=1$ gives rays
$(1,1),(0,-1),(-1,0)$, the fan of $\PP^2$ after a lattice change.
Figure~\ref{fig:wedge} shows the corresponding interval and triangle.
More generally, wedging the fan of $\PP^n$ gives the fan of $\PP^{n+1}$.
\begin{figure}[H]
\centering
\begin{tikzpicture}[x=1cm,y=1cm,>=Stealth,font=\small]
\begin{scope}[shift={(0,0)}]
\draw[fanedge,->] (-.3,0)--(2.1,0) node[right] {$x$};
\draw[markedwall] (0,0)--(1.6,0);
\fill (0,0) circle (1.6pt);
\fill (1.6,0) circle (1.6pt);
\node[above=3pt] at (0,0) {$0$};
\node[above=3pt] at (1.6,0) {$1$};
\node at (.8,-2.5) {$P=[0,1]$};
\end{scope}
\draw[->,line width=.7pt] (2.6,-.6)--(3.3,-.6);
\begin{scope}[shift={(4,0)}]
\fill[fanmark!8] (0,0)--(1.6,0)--(1.6,-1.6)--cycle;
\draw[fanedge,->] (-.2,0)--(2.1,0) node[right] {$x$};
\draw[fanedge,->] (0,-1.85)--(0,.45) node[above] {$t$};
\draw[markedwall] (0,0)--(1.6,0)--(1.6,-1.6)--cycle;
\foreach \x in {.4,.8,1.2}
  \draw[fanedge,dashed] (\x,0)--(\x,-\x);
\fill (0,0) circle (1.8pt);
\node[above left=1pt] at (0,0) {$0$};
\node[above=3pt] at (1.6,0) {$1$};
\node[right=3pt] at (1.6,-1.6) {$(1,-1)$};
\node at (.8,-2.5) {$0\le x\le1,\ -x\le t\le0$};
\end{scope}
\begin{scope}[shift={(9.2,-.5)}]
\draw[markedwall,->] (0,0)--(1,1) node[above right] {$v_+$};
\draw[markedwall,->] (0,0)--(0,-1.15) node[right] {$v_-$};
\draw[markedwall,->] (0,0)--(-1.3,0) node[above] {$u$};
\fill (0,0) circle (1.8pt);
\node at (0,-2) {Fan of $\PP^2$};
\end{scope}
\end{tikzpicture}
\caption{The wedge polytope for the fan of $\PP^1$.
The interval fiber over $x$ has endpoints $t=-x$ and $t=0$ and
collapses at the facet $x=0$. The triangle has inward normals
$v_+=(1,1)$, $v_-=(0,-1)$, and $u=(-1,0)$, as in~\eqref{eq:wedge-polytope}.}
\label{fig:wedge}
\end{figure}

\begin{lemma}\label{lem:wedge}
The wedge is complete and simplicial, with one more dimension and one
more ray. Its maximal-cone multiplicities are those of the corresponding
original cones. Its quotient star at $v_-$ is $\Sig$, and it is projective
if and only if $\Sig$ is projective. For every $\mu\ge1$,
\[
 R_\mu(W_v\Sig)=R_\mu(\Sig).
\]
If $\Sig$ is smooth, then
$\ell_{\min}(W_v\Sig)=\ell_{\min}(\Sig)$.
\end{lemma}
\begin{proof}
\emph{Fan structure and projectivity.}
Write $\widetilde u=(u,0)$ for a horizontal lift. Apply the
piecewise-linear homeomorphism
$(x,t)\mapsto(x+\max(t,0)v,t)$ to the product of $\Sig$ with the fan
of $\PP^1$. Over each original maximal cone containing $v$, merge the
two resulting cones. Their union is the stated wedge cone, and the
ray $(v,0)=v_++v_-$ subdivides it back into those two cones. These
mergings respect face intersections: the intersection of two resulting
cones is generated by their common lifted rays and whichever of
$v_+,v_-$ they share. This follows by taking their intersections in
each closed half-space $t\ge0$ and $t\le0$, where the shear is linear.
Each such generating subset spans a face of both simplicial cones.
This gives the wedge fan. It is
complete: its slices at height $t>0$ are the translates $tv+\sigma$,
and at height $t<0$ they are the original cones $\sigma$.
Determinant expansion in the last coordinate proves the multiplicity
assertion and hence preservation of smoothness. The quotient star is
immediate from the cone list.

If $\Sig$ is the normal fan of the rational polytope
$P=\{x:\langle x,u\rangle\ge h_u\}$, its wedge is the normal fan of
\begin{equation}\label{eq:wedge-polytope}
 \{(x,t):\langle x,u\rangle\ge h_u\ (u\ne v),\quad
                 \langle x,v\rangle+t\ge h_v,\quad -t\ge0\}.
\end{equation}
The fiber over $x\in P$ is the interval
$[h_v-\langle x,v\rangle,0]$, which collapses exactly on the facet
with normal $v$. This gives the asserted cone list for its normal fan.
Conversely, the invariant divisor of $v_-$ in a projective
wedge is projective and has fan $\Sig$.

\smallskip\noindent
\emph{Ray counts.} The upper bound comes from wedging a
refinement; the lower bound comes from taking its quotient star.
If $\Phi\preceq\Sig$, then $W_v\Phi\preceq W_v\Sig$. Over a parent
cone containing $v$, this follows from $(v,0)=v_++v_-$; over other
cones it is immediate. Wedging a minimizing projective refinement
therefore preserves its new-ray count and multiplicity cap. This gives
the upper bound for $R_\mu$.

For the lower bound, take any projective simplicial refinement $\Psi$
of $W_v\Sig$ with multiplicities at most $\mu$. Its quotient star at
the unchanged ray $v_-$ is a complete projective refinement of $\Sig$.
Every new quotient ray is the image of a new ray of $\Psi$, since the
old wedge rays project to old rays of $\Sig$. It therefore has no more
new rays. If a containing maximal cone has generators $v_-,r_1,\ldots,r_n$
and the images of $r_i$ are $d_i\bar r_i$ with $\bar r_i$ primitive,
then
\[
 \mult\langle v_-,r_1,\ldots,r_n\rangle
       =\Bigl(\prod_i d_i\Bigr)\mult\langle\bar r_1,\ldots,\bar r_n\rangle.
\]
Thus quotient multiplicities also satisfy the cap, proving equality.

\smallskip\noindent
\emph{Lifting ordinary blow-ups.}
Assume $\Sig$ is smooth. For a center
$\tau=\langle u_1,\ldots,u_d\rangle$, $d\ge2$, put
$s=\sum_i u_i$. Let $\widetilde\tau$ be the wedge cone over $\tau$:
the cone spanned by the horizontal lifts of its generators if
$v\notin\tau$, and the cone with $v$ replaced by $v_+,v_-$ if
$v\in\tau$. In both cases the sum of its primitive generators is
$\widetilde s=(s,0)$, since $v_++v_-=\widetilde v$. We compare
subdivision at this lifted center with wedging the subdivision
at $\tau$.

If $v\notin\tau$, the lifted center and its sum ray are
\[
 \widetilde\tau=\langle\widetilde u_1,\ldots,\widetilde u_d\rangle,
 \qquad \widetilde s=\sum_i\widetilde u_i=(s,0).
\]
Let $K$ be any wedge maximal cone containing $\widetilde\tau$, and let
$G(K)$ denote its primitive generators. Its children are
\[
 \left\langle\widetilde s,\ G(K)\setminus\{\widetilde u_i\}\right\rangle,
 \qquad 1\le i\le d.
\]
These are exactly the wedge cones over the children obtained by omitting
$u_i$ in the original cone: each retains both $v_+,v_-$ if its parent
contains both, and otherwise retains its parent's single sign ray.

If $v\in\tau$, write $u_1=v$. In this case
\[
 \widetilde\tau=\langle v_+,v_-,\widetilde u_2,\ldots,\widetilde u_d\rangle,
 \qquad \widetilde s=v_++v_-+\sum_{i=2}^d\widetilde u_i=(s,0).
\]
For an original maximal cone
$\sigma=\langle v,u_2,\ldots,u_d,B\rangle$, where $B$ lists its
remaining generators, the wedge cone has the following children:
\[
\begin{aligned}
 &\langle\widetilde s,v_-,\widetilde u_2,\ldots,\widetilde u_d,\widetilde B\rangle,
 \quad
 \langle\widetilde s,v_+,\widetilde u_2,\ldots,\widetilde u_d,\widetilde B\rangle,\\
 &\left\langle\widetilde s,v_+,v_-,
       \{\widetilde u_j:2\le j\le d,\ j\ne i\},\widetilde B\right\rangle
       \quad(2\le i\le d).
\end{aligned}
\]
The first two are the two wedge extensions of the original child
omitting $v$; the others wedge the children omitting $u_i$.
Cones not containing the center are unchanged in both constructions.
Thus subdivision commutes with wedging in both cases, and an optimal
ordinary sequence lifts with the same number of steps.

\smallskip\noindent
\emph{Projecting ordinary blow-ups.}
Conversely, follow the quotient star at the unchanged ray $w=v_-$
through an ordinary sequence on $W_v\Sig$. Write $\bar u$ for the
image of $u$ in $(N\oplus\ZZ)/\ZZ w$. In a smooth cone containing
$w$, the other generators project to primitive generators of a smooth
quotient cone. There are four cases for a center $\eta$.

If $w\notin\eta$ but $\langle w,\eta\rangle$ is a cone, write
$\eta=\langle u_1,\ldots,u_d\rangle$ and consider a containing maximal
cone $K=\langle w,u_1,\ldots,u_d,B\rangle$. The center, sum ray, and
children on the quotient are
\[
\begin{gathered}
 \bar\eta=\langle\bar u_1,\ldots,\bar u_d\rangle,
 \qquad \bar s=\sum_{i=1}^d\bar u_i,\\
 \left\langle\bar s,\{\bar u_j:j\ne i\},\bar B\right\rangle
 \quad(1\le i\le d).
\end{gathered}
\]
All children of $K$ still contain $w$, so this is one ordinary blow-up
of the quotient at $\bar\eta$.

If $w\in\eta$ and $\dim\eta=d\ge3$, write
$\eta=\langle w,u_2,\ldots,u_d\rangle$. Then
\[
\begin{gathered}
 s=w+\sum_{i=2}^d u_i,\qquad
 \bar\eta=\langle\bar u_2,\ldots,\bar u_d\rangle,
 \qquad \bar s=\sum_{i=2}^d\bar u_i,\\
 \left\langle\bar s,\{\bar u_j:2\le j\le d,\ j\ne i\},\bar B\right\rangle
 \quad(2\le i\le d)
\end{gathered}
\]
are the quotient center, ray, and children. The child omitting $w$
is absent from its star; the other children give one ordinary blow-up
at the $(d-1)$-dimensional quotient center.

If $\eta=\langle w,u\rangle$, the sum ray $s=w+u$ projects to the
existing ray $\bar u$. Of the two children of
$\langle w,u,B\rangle$, only $\langle w,s,B\rangle$ contains $w$,
and its quotient is the unchanged cone $\langle\bar u,\bar B\rangle$.
Thus this case induces the identity. Finally, if no cone contains
both $w$ and $\eta$, no cone of the star is altered and the quotient
is again unchanged.

The final quotient is projective and smooth, and omitting these identity
steps leaves an ordinary sequence with no more steps than the given
one. This proves the reverse inequality for $\ell_{\min}$.
\end{proof}

\begin{proof}[Proof of Theorem~\ref{thm:sharp}]
Set $\Sig_3(a,b)=\Sig(a,b)$ and iterate the wedge $n-3$ times,
at $v_6$, then at $v_6^+$, then at $(v_6^+)^+$, and so on.
The number of rays becomes $n+5$, so the Picard number remains five.
For fixed $n$ the cone incidences do not depend on $a,b$.
Theorem~\ref{thm:threefold} and Lemma~\ref{lem:wedge} give the formulas
for $R_\mu$ and $\ell_{\min}$. The two smooth chains constructed in
dimension three lie in $12$ and $58$; the shorter is optimal.
Neither uses the distinguished wedge ray.
Their lifted centers therefore remain two-cones. Finally, a complete
simplicial $n$-fan with $n+5+r$ rays has Picard rank $5+r$.
\end{proof}

\begin{corollary}\label{cor:consequences}
In every fixed dimension $n\ge3$, neither original Picard number nor
original combinatorial type bounds the minimum number of new rays in
a smooth projective refinement. At Picard number five there are binary
inputs of length $\Theta(\log(k+1))$ requiring at least $k$ new ray
entries. Among projective simplicial refinements of $\Sig_n(a,b)$
with maximal multiplicity at most $\mu$, the minimum Picard rank is
$5+\lceil\min(a,b)/\mu\rceil$.
\end{corollary}
\begin{proof}
Use $a=b=k$ for the first two assertions. For fixed $n$, only two
original coordinates vary, and wedging introduces only fixed additional
coordinates and cone incidences. The last assertion again uses
$\rho=|\Phi(1)|-n$.
\end{proof}

In dimension three, Picard number five is minimal for such unbounded
cost: complete smooth toric varieties of Picard number at most three
are projective \citep{KleinschmidtSturmfels}, and Oda's classification
contains only one nonprojective smooth complete threefold of Picard
number four \citep[\S9]{Oda1978}; see also
\citet[Remark~1.6 and \S4]{FujinoSato}.
Its cost is one, as recalled in Section~\ref{sec:symmetry}.

\subsection{Product additivity}
\begin{proposition}\label{prop:products}
For smooth complete fans $\Sig,\Delta$,
\[
 R_1(\Sig\times\Delta)=R_1(\Sig)+R_1(\Delta),\qquad
 \ell_{\min}(\Sig\times\Delta)=\ell_{\min}(\Sig)+\ell_{\min}(\Delta).
\]
\end{proposition}
\begin{proof}
The upper bounds follow by taking products of minimizing refinements
and by lifting minimizing blow-up sequences from the two factors.
For the lower bounds, restrict a smooth projective refinement of the
product to each factor subspace. These restrictions are subfans:
within an original product
cone, a functional positive on the other factor cone away from zero
cuts out the factor subspace as a face. The restricted fans are complete
and smooth, and restriction of a strictly upper convex support function proves
projectivity. A new ray can lie in at most one factor subspace, so the
total number is at least $R_1(\Sig)+R_1(\Delta)$.

A sum ray of a blow-up lies in a factor subspace only if its entire
center cone does, since its other coordinates belong to a strongly
convex cone. Such a step restricts to the same ordinary blow-up on
that factor; other steps leave its subfan unchanged. Thus each step
contributes to at most one of the two induced projectivizing sequences,
proving the length lower bound as well.
\end{proof}
Maximal-cone multiplicities multiply under products, so this argument
does not prove additivity for $R_\mu$ when $\mu>1$.

\section{Canonical projectivization}\label{sec:canonical}

\subsection{Smooth input: an intrinsic ideal}\label{subsec:ideal}
We seek an ideal whose principalization forces every wall character
to have one sign on each cone. On an affine chart, the ideal associated
to one character will have two monomial generators; making it principal
forces one generator to divide the other.

Work first over $\QQ$, and fix the principalization algorithm in
\citet[Theorems~1.0.1 and 2.4.2]{WResolution}. On the smooth split toric
variety $X=X_\Sig$, let $D_\rho=V(\rho)$ be the invariant divisor
of a ray $\rho$, with primitive generator $u_\rho$. For $m\in M$,
write $\chi^m$ for the corresponding character of the dense torus.
Its zero and pole divisors are given by
\[
 \operatorname{div}(\chi^m)=D_m^+-D_m^-,\qquad
 D_m^\pm=\sum_{\rho\in\Sig(1)}
                  \max(\pm\langle m,u_\rho\rangle,0)D_\rho.
\]
These are effective Cartier divisors. Define
\begin{equation}\label{eq:intrinsic-ideal}
 \cJ_{[m]}=\mathcal O_X(-D_m^+)+\mathcal O_X(-D_m^-)
       \subseteq\mathcal O_X,\qquad
 \cJ_\Sig=\prod_{[m]\in S_\Sig}\cJ_{[m]}.
\end{equation}
Changing $m$ to $-m$ interchanges the summands; the product requires
no ordering. Thus $\cJ_\Sig$ is intrinsic, torus-invariant, and equal
to the unit ideal on the dense torus. More precisely, on the affine
toric chart $U_\sigma$ it is the unit ideal exactly when $\sigma$ is
adapted to every wall character, equivalently when $h_{Z_\Sig}$ is
linear on $\sigma$. Indeed, each factor is a pair of monomials with
disjoint variable supports, so it is the unit ideal exactly when one
generator is a unit, that is, when $m$ has one sign on $\sigma$.
Thus the ideal records precisely the failure of sign-adaptation.

The principalization algorithm keeps a record of exceptional divisors,
called its \emph{exceptional history}; we begin with an empty record.
A \emph{formal divisorial step} records the blow-up of a smooth divisor
in this history, although its morphism is an isomorphism. By a
\emph{geometric center} we mean the actual smooth closed subscheme
blown up, as distinct from the accompanying history data.

\begin{construction}[Canonical smooth projectivization]\label{con:smooth}
Principalize $\cJ_\Sig$ by the fixed algorithm. Retain the resulting
sequence of fans and smooth
invariant centers. After the algorithm has finished, delete the
codimension-one components of its centers, whose blow-up is an
isomorphism, and omit the resulting identity steps.
Over any field, realize this same fan sequence by the corresponding
ordinary toric blow-ups.
\end{construction}

\begin{proposition}\label{prop:principalization}
Construction~\ref{con:smooth} terminates with a smooth projective fan
and is natural under lattice isomorphisms. Every nontrivial geometric
center is a disjoint union of invariant smooth orbit closures of
codimension at least two.
\end{proposition}
\begin{proof}
We first show that principalization stays within the toric category,
then use invertibility of the ideal to prove adaptation and projectivity.
Canonical principalization gives a finite sequence of blow-ups with
smooth centers and smooth intermediate varieties, and makes
$\cJ_\Sig\mathcal O_Y$ invertible on the final variety $Y$.
The dense torus $\mathbb T$ preserves $\cJ_\Sig$. Functoriality under smooth
morphisms \citep[Theorem~1.0.1]{WResolution}, applied to the action and
projection maps $\mathbb T\times X\rightrightarrows X$, therefore lifts the
torus action through the entire sequence: both maps are smooth and pull
back the invariant ideal to the same ideal. Functoriality under lattice
isomorphisms gives the corresponding naturality. Each center is thus
torus-invariant. A smooth invariant closed subscheme in a smooth toric
variety is a disjoint union of smooth orbit closures: its irreducible
components are orbit closures, and distinct components of a smooth
scheme cannot meet. Blowing up these disjoint centers gives ordinary
sum-ray subdivisions, so all intermediate varieties are smooth toric.

To prove adaptation, first note that each factor
$\cJ_{[m]}\mathcal O_Y$ is invertible. Indeed, in a local
domain a principal nonzero product $IJ=(f)$ makes $I$ invertible as
a fractional ideal, with inverse $J/f$. Apply this to each factor of
the product in~\eqref{eq:intrinsic-ideal}.

On an original smooth affine chart the two generators of $\cJ_{[m]}$
are monomials whose ratio is $\chi^m$. Their pullbacks remain monomials
on every final smooth chart. At its torus fixed point, a principal ideal
generated by two monomials has one of those monomials dividing the
other. Indeed, write the monomials as $fa,fb$, where $f$ generates the
ideal. Then $(a,b)$ is the unit ideal in the local ring, so one of
$a,b$ is a unit. Monomial divisibility compares exponents coordinate
by coordinate. Their differences are $\langle m,u_\rho\rangle$ for
the rays of the final cone, so these values are either all nonnegative
or all nonpositive.
The final fan $\Gam$ is therefore adapted to every normal in $S_\Sig$.

The morphism $Y\to X$ is projective as a composite of blow-ups.
Theorem~\ref{thm:selective}, selecting all walls, now gives a strictly
upper convex support function
\[
 H=-\sum_{[m]\in S_\Sig}|m|+\varepsilon g
\]
for sufficiently small rational $\varepsilon>0$, where
$g\in\PL_\QQ(\Gam)$ is strictly upper convex relative to $\Sig$.
Hence $Y$ is projective.

The construction of the ideal and the fixed principalization sequence
commute with lattice isomorphisms, so the fan sequence does also.
Smoothness, disjointness of centers, the sum-ray rule, and the rational
support function depend only on lattice data. The same sequence over
any field thus has all the asserted properties. A smooth divisorial
blow-up is an identity. Removing its geometric step after the algorithm,
or removing divisorial components from a mixed disjoint center, changes
neither the fan nor the final morphism.
\end{proof}

The principalization history is retained while the algorithm runs;
deleting identity steps beforehand is not part of the construction.
The passage from $\QQ$ to other fields uses only the resulting lattice
data; it makes no assertion about principalization for general varieties
in positive characteristic. Our naturality statement is under lattice
isomorphisms. Since the ideal uses the walls of the whole
input, compatibility with restriction to open toric subvarieties is
not asserted.

The same proof gives a selective version. For a specified finite set
$S$ of unoriented characters, principalize $\prod_{[m]\in S}\cJ_{[m]}$.
The result is smooth, $S$-adapted, and projective over the original fan.
If the walls with normals in $S$ form an obstruction transversal, it
is projective. This is canonical for the pair $(\Sig,S)$, and intrinsic
to $\Sig$ if $S$ is assigned intrinsically. No optimality or
codimension-two restriction is claimed for these centers.

\subsection{Arbitrary input: intrinsic labelled resolution}\label{subsec:labelled}
For nonsmooth input we use a direct construction on fans. We first
obtain a projective simplicial refinement, then lower the multiplicities
of its cones. The lattice-point method for this second step is standard
\citep[\S11.1]{CLS}. The role of the labels is to specify its choices
without a lattice basis and to ensure that simultaneous subdivisions
commute.

For each nonzero arrangement cone $\tau$, choose the ray
\[
 b_\tau=\prim\left(\sum_{\rho\in\tau(1)}u_\rho\right)
\]
and give it label $\dim\tau$. The cones generated by the $b_\tau$
along strict chains of arrangement faces form the \emph{barycentric
subdivision} $\Phi_0=\operatorname{Bar}(\cA_\Sig)$.
Consequently $\Phi_0$ is complete and
simplicial, and the rays in every cone have distinct labels.
It is projective: barycentric subdivision is realized by star
subdivisions in decreasing original face dimension. At a fixed
dimension these subdivisions commute. Indeed, a current cone consists
of one unprocessed face and a chain of interior rays of larger faces;
it cannot contain two distinct unprocessed faces of the same dimension.
Thus no ordering within a dimension is needed.

We now resolve $\Phi_0$ in rounds, stopping as soon as the fan is smooth.
At round $t\ge1$, choose the smallest dimension $d$ of a nonsmooth
cone. For each such cone, list its ray labels in increasing order;
select all cones with the lexicographically least label tuple. In each
selected cone $\tau=\langle u_1,\ldots,u_d\rangle$, order the generators
by their labels. Its \emph{half-open fundamental parallelepiped} is
\[
 \Pi_\tau=\left\{\sum_{i=1}^d a_i u_i:0\le a_i<1\right\}.
\]
Among the nonzero lattice points of $\Pi_\tau$, choose the one with
lexicographically least coefficient tuple $(a_1,\ldots,a_d)$.
Subdivide at the rays through these points and give them the common fresh label
$n+t$. Repeat until smooth.

For a local example, suppose that the wall arrangement has a maximal
cone $K$ equal to the cone over the square with vertices
$(\pm3,\pm3,1)$; the labels below are those that the global construction
assigns to the faces of $K$. One vertex--edge--face chain gives the
barycentric rays
\[
 u=(3,3,1)\ [1],\qquad e=(0,3,1)\ [2],\qquad c=(0,0,1)\ [3],
\]
where brackets indicate labels. The three two-cones spanned by pairs
of these rays have multiplicity $3$, so the first selected label tuple
is $(1,2)$, and for $\tau=\langle u,e\rangle$,
\[
 \Pi_\tau\cap\ZZ^3=\{0,\ (1,3,1),\ (2,3,1)\},\qquad
 (1,3,1)=\tfrac13u+\tfrac23e,\qquad (2,3,1)=\tfrac23u+\tfrac13e.
\]
The coefficient tuples $(\tfrac13,\tfrac23)$ and $(\tfrac23,\tfrac13)$
are ordered by the labels of $u$ and $e$, so the first round inserts
the ray $s=(1,3,1)$ with label $4$, simultaneously with the
corresponding rays in all other cones having tuple $(1,2)$. The
children $\langle s,e\rangle$ and $\langle u,s\rangle$ of $\tau$, which
omit $u$ and $e$ respectively, have multiplicities $\tfrac13\cdot3=1$
and $\tfrac23\cdot3=2$.
In general, the least tuple takes the smallest coefficient that occurs
on the ray of least label, then the smallest coefficient on the ray of
next label among the points achieving it, and so on. The labels thus
specify the coefficient order without choosing coordinates.

\begin{proposition}\label{prop:labelled}
This procedure terminates and gives a smooth projective refinement
natural under lattice isomorphisms. If $\Phi_0$ has $F$ maximal cones
and maximum multiplicity $D$, the resolution stage uses at most
\[
 F\frac{n^{D-1}-1}{n-1}
\]
individual ray insertions and finishes with at most $Fn^{D-1}$ maximal
cones.
\end{proposition}
\begin{proof}
We verify that each round is well defined, decreases multiplicities,
and respects lattice isomorphisms.
The parallelepiped contains $\mult(\tau)>1$ lattice points, one for
each coset of $\sum_i\ZZ u_i$ in $N\cap\Span_\RR\tau$, so a
nonzero point is available. Every proper face is smooth by minimality
of $d$; a nonzero point cannot have a zero coefficient, since a smooth
face's half-open parallelepiped contains only zero. The selected point
is already primitive: dividing by a larger positive integer would give
a smaller coefficient tuple in the same set. Write it as
$s_\tau=\sum_i a_i u_i$, with $0<a_i<1$.

Two selected cones cannot lie in a common cone. They have the same
label tuple, while each containing cone has at most one ray with any
given label. Their subdivisions therefore commute, and each affected
cone receives only one fresh label. Distinctness of labels is preserved.
If $\sigma$ contains $\tau$, its child omitting $u_i$ has multiplicity
\[
 \mult(\sigma_i)=a_i\mult(\sigma)<\mult(\sigma).
\]
These multiplicities are positive integers, so they strictly decrease
along every chain of subdivisions. To bound the whole process, form a
rooted tree for each initial maximal cone, joining a cone to the maximal
subcones that replace it. Each vertex has at most $n$ children, and
each branch has at most $D-1$ replacements. There are at most
$Fn^{D-1}$ leaves and
$F(1+n+\cdots+n^{D-2})$ internal nodes. Every ray insertion replaces
at least one maximal cone, proving termination and both bounds
(including zero insertions for $D=1$).

All subdivisions are projective, so the terminal fan is projective
over the projective arrangement fan and hence projective. Its morphism
to $X_\Sig$ is also projective: its graph is closed because $X_\Sig$
is separated, and projection from the product with a projective variety
is projective. Lattice isomorphisms preserve dimensions, multiplicities,
labels, coefficient tuples, and the chosen lattice points. They therefore
transport every batch and its output, proving naturality.
\end{proof}

\begin{proof}[Proof of Theorem~\ref{thm:canonical}]
For smooth input use Construction~\ref{con:smooth}; otherwise use
Proposition~\ref{prop:labelled}. This intrinsic choice preserves
naturality. Both outputs are smooth projective
refinements, and their maps to the input are projective and birational.
For smooth input Proposition~\ref{prop:principalization} supplies the
stronger assertion about all intermediate varieties and centers.
\end{proof}

The bound in Proposition~\ref{prop:labelled} starts after the barycentric
arrangement, whose size and multiplicities may exceed those of the
input. It is not a bound in terms of the original combinatorial type.
For nonsmooth input no ordinary smooth-center factorization is claimed.
One can make the assignment leave already smooth projective fans
unchanged by the intrinsic preliminary test of feasibility of the strict
bend inequalities.

\section{Oda's threefold: selectivity and symmetry}\label{sec:symmetry}
Let $\Sig_{\mathrm O}$ be the fan of Oda's smooth nonprojective threefold of
Picard number four \citep[\S9, case {$[7\text{-}5]$}]{Oda1978},
\citep[Example~4.1]{FujinoSato}. Its rays are
\[
 v_1=e_1,\ v_2=e_2,\ v_3=e_3,\ v_4=(-1,-1,-1),\
 v_5=(-1,-1,0),\ v_6=(0,-1,-1),\ v_7=(-1,0,-1),
\]
and its maximal cones are
\[
 123,\ 127,\ 136,\ 167,\ 235,\ 257,\ 356,\ 456,\ 457,\ 467.
\]

Figure~\ref{fig:oda} marks the three walls supporting the obstruction.
\begin{figure}[htbp]
\centering
\begin{tikzpicture}[x=1cm,y=1cm]
\coordinate (o1) at (0,2.7);
\coordinate (o2) at (-2.6,-1.5);
\coordinate (o3) at (2.6,-1.5);
\coordinate (o4) at (0,0);
\coordinate (o5) at (0,-.95);
\coordinate (o6) at (1.02,.60);
\coordinate (o7) at (-1.02,.60);
\fill[black!3] (o1)--(o2)--(o3)--cycle;
\foreach \a/\b in {1/2,1/3,2/3,1/7,2/7,1/6,3/6,6/7,2/5,3/5,5/7,5/6,4/5,4/6,4/7}
  \draw[fanedge] (o\a)--(o\b);
\foreach \a/\b in {1/7,2/5,3/6}
  \draw[markedwall] (o\a)--(o\b);
\foreach \i in {1,...,7}
  \node[fanvertex] at (o\i) {$\i$};
\end{tikzpicture}
\caption{The triangulated sphere of Oda's fan, with the unbounded face
representing $123$. The thick edges $17,25,36$ are cycled by
$R(x,y,z)=(z,x,y)$ and satisfy
$\delta_{17}+\delta_{25}+\delta_{36}=0$.
No triangle contains two marked edges, so their orbit closures are
pairwise disjoint. Positions are schematic.}
\label{fig:oda}
\end{figure}

The obstruction to projectivity is concentrated on three walls:
$\delta_{17}+\delta_{25}+\delta_{36}=0$.
The values $(0,0,0,-10,-8,-8,-8)$ have zero bends exactly on these
walls and positive bends elsewhere. Hence every nonnegative dependence
is supported on these three walls. Their bend rows have rank two. Writing $e_\omega$ for the coordinate vector indexed by wall $\omega$,
we obtain
\[
 \cC(\Sig_{\mathrm O})=\RR_{\ge0}(e_{17}+e_{25}+e_{36}).
\]
The ray values $(0,0,0,-10,-9,-8,-7)$ give positive bends except at
$25$, where the bend is $-2$. Its normal is the coordinate character $z$,
and the only
$z$-bad two-cone is $36$. One ordinary blow-up inserts $(0,-1,0)$;
the values
\[
 (0,0,-8,-28,-18,-24,-22,-15)
\]
have positive integral bends on the resulting fan, with minimum one.
Thus $R_1(\Sig_{\mathrm O})=\ell_{\min}(\Sig_{\mathrm O})=1$.
Full basis-ordered adaptation instead takes $25$ steps; its counts
and a singular comparison are in Online Resource~1.

\begin{proposition}[Individual centers and simultaneous centers]\label{prop:symmetry}
There is no rule assigning every smooth complete threefold fan a
projectivizing ordered sequence of ordinary blow-ups at individual
invariant orbit closures that is natural under lattice isomorphisms.
Allowing canonical exceptional history and formal divisorial steps
does not change this conclusion. Oda's fan nevertheless has an intrinsic
smooth projective refinement adding three rays, and every projective
refinement invariant under $R(x,y,z)=(z,x,y)$ adds at least three.
\end{proposition}
\begin{proof}
Naturality forces each individual center to be preserved by every
automorphism of the input. We apply this requirement to $R$.
On $\Sig_{\mathrm O}$, $R$ fixes ray $4$ and cycles $(1,2,3)$ and
$(5,6,7)$. There is no invariant two-cone and the only invariant
three-cone is $123$. Its blow-up inserts $p=(1,1,1)$ and replaces
$123$ by three cyclically permuted cones. The new fan has no invariant
three-cone. An invariant two-cone would have two fixed rays, since $R$
has order three. But its only fixed rays are $p$ and $v_4$, which are
opposite and cannot span a strongly convex cone. The three critical walls and
their adjacent cones are unchanged, and their bend relation still
obstructs projectivity. Thus the only possible nontrivial first
individual center cannot finish the construction, and no second
individual center is available. Divisorial steps change no fan;
functorially assigned history remains $R$-invariant.

In contrast, the orbit closures of $17,25,36$ are pairwise disjoint,
since no cone contains two of these walls. Their union is intrinsic,
since every automorphism preserves the obstruction cone and hence the
support $\{17,25,36\}$ of its unique ray. Blowing up this
smooth disconnected center inserts
\[
 s_{17}=(0,0,-1),\qquad s_{25}=(-1,0,0),\qquad s_{36}=(0,-1,0).
\]
The output is smooth with $f$-vector $(10,24,16)$ and admits a strictly
upper convex support function with ray values
\[
 (0,0,0,-7,-5,-5,-5,-4,-4,-4).
\]
For values $(0,0,0,B,C,C,C,D,D,D)$, the wall types and bends are
shown in Table~\ref{tab:oda-symmetric-bends}. Substitution of
$(B,C,D)=(-7,-5,-4)$ gives $1,1,4,3,1$, respectively.
\begin{table}[htbp]
\centering
\caption{Bends after the simultaneous blow-up of $17,25,36$.
The twelve walls in the second row each contain one of
$s_{17},s_{25},s_{36}$.}
\label{tab:oda-symmetric-bends}
\begin{tabular}{@{}lcc@{}}
\toprule
Wall(s)&Bend&Value\\
\midrule
$56,57,67$&$-B+2C-D$&$1$\\
Walls containing a new ray&$-C+D$&$1$\\
$12,13,23$&$-D$&$4$\\
$16,27,35$&$C-2D$&$3$\\
$45,46,47$&$2B-3C$&$1$\\
\bottomrule
\end{tabular}
\end{table}

Every ray orbit under $R$ has size one or three. The only missing
fixed ray is the positive diagonal through $p$; the negative diagonal
is already ray $4$. With fewer than three new rays, an invariant
refinement can only insert $p$, necessarily giving the nonprojective
subdivision above. This proves the sharp symmetric lower bound.
\end{proof}

The center of this simultaneous blow-up has three connected components.
Blowing them up in any order gives the same fan and counts as three
individual steps in Definition~\ref{def:cost}. A canonical construction
can instead give all three components the same stage label, without
ordering them. This is
consistent with the disconnected centers in functorial factorization
\citep[remark after Theorem~0.3.1]{AKMW}.

Exceptional history explains the apparent symmetry breaking in
resolution. W{\l}odarczyk permits formal divisorial steps and orders
exceptional divisors while
allowing disjoint irreducible components within one divisor
\citep[p.~782 and Definition~2.1.1]{WResolution}.
Such history can distinguish local models whose unmarked forms look
identical. It cannot distinguish components exchanged by an automorphism
of the original input that preserves the whole construction.
The positive result of Section~\ref{sec:canonical} therefore uses
disconnected smooth centers; it does not remove the obstruction to
ordering their individual components. Nor does it prove that every
canonical smooth projectivization can use only codimension-two centers.

\section{Exact verification}\label{sec:verification}
The driver \texttt{verify\_all.py} in Online Resource~2 runs two exact-arithmetic
suites. The first suite, \texttt{verify\_certificates.py}, checks the smooth
family certificates and obstruction relations for all $144$ pairs
$1\le a,b\le12$. It also checks Oda's certificates, the singular
comparison, the $25$- and $42$-step runs, and the symmetric three-curve
blow-up. Finally, it checks the labelled arrangement construction on
three examples and after nine lattice isomorphisms. These last outputs
belong to the direct labelled construction, not to the principalization
branch.

The second suite, \texttt{verify\_wedges.py}, checks $768$ parameter/cap
combinations for the chains with bounded multiplicity, including every
intermediate multiplicity and the final strict rational support functions. It checks $80$ wedge examples in dimensions
four through seven, their quotient stars, Picard numbers and support
certificates, and $230$ wedge/subdivision commutation cases.
All checks use integer or rational arithmetic. The principalization
algorithm itself is invoked through the cited theorem and is not
implemented by these scripts. The proofs, rather than the finite
checks, establish the all-parameter bounds, termination, and naturality.

\section*{Notation index}
\begin{center}
\small
\begin{tabular}{@{}llr@{}}
\toprule
Symbol&Meaning&Introduced in\\
\midrule
$R_\mu(\Sig)$&Minimum new rays under multiplicity cap $\mu$&Definition~\ref{def:cost}\\
$\ell_{\min}(\Sig)$&Minimum number of individual ordinary blow-ups&Definition~\ref{def:cost}\\
$\PL_\QQ(\Sig)$&Rational support functions on $\Sig$&\S\ref{subsec:arrangement}\\
$S_\Sig$&Unoriented primitive wall characters (wall normals)&\S\ref{subsec:arrangement}\\
$\cA_\Sig$&Wall arrangement fan&\S\ref{subsec:arrangement}\\
$Z_\Sig$&Zonotope $\sum_{[m]\in S_\Sig}[-m,m]$&\eqref{eq:zonotope}\\
$h_{Z_\Sig}$&Its support function $-\sum_{[m]\in S_\Sig}|m|$&\eqref{eq:zonotope}\\
$\delta_\omega$&Bend at the wall $\omega$&\S\ref{subsec:transversals}\\
$\cC(\Sig)$&Cone of nonnegative dependencies among bends&Definition~\ref{def:transversal}\\
$\cV,\ S_\cV$&Obstruction transversal and its wall characters&Definition~\ref{def:transversal}\\
$W_v\Sig$&Toric wedge at the ray $v$&\S\ref{subsec:wedge}\\
$\cJ_{[m]},\ \cJ_\Sig$&Character ideal and product over wall characters&\eqref{eq:intrinsic-ideal}\\
$\Phi_0=\operatorname{Bar}(\cA_\Sig)$&Barycentric subdivision of the wall arrangement&\S\ref{subsec:labelled}\\
$b_\tau$&Barycentric ray of the arrangement cone $\tau$&\S\ref{subsec:labelled}\\
Labels&$\dim\tau$ on $b_\tau$; $n+t$ on the rays inserted in round $t$&\S\ref{subsec:labelled}\\
$\Pi_\tau$&Half-open fundamental parallelepiped of $\tau$&\S\ref{subsec:labelled}\\
\bottomrule
\end{tabular}
\end{center}

\section*{Statements and declarations}
\noindent\textbf{Funding.} The author received no funds, grants, or other
support for this work.

\smallskip
\noindent\textbf{Competing interests.} The author declares no conflict
of interest.

\smallskip
\noindent\textbf{Use of AI tools.} This work was a collaboration between
the author and frontier AI models accessed through individual consumer
services: OpenAI's Sol and Astra, Anthropic's Opus and Fable, and
Google's Gemini~3.1~Pro and Gemini~3.8~Flash. The models contributed
throughout the development, refinement, organization, and writing of the
paper, including the proofs and the accompanying verification code. The
author directed the research, worked with the models in developing the
arguments and results, and shaped their final presentation. The author
takes full responsibility for the content of the manuscript.

\smallskip
\noindent\textbf{Data availability.} The detailed certificates are in
Online Resource~1. The verification driver, both component scripts, the
exact-arithmetic library \texttt{toric\_exact\_lib.py}, and the run log
are in Online Resource~2. The supplement is appended to this arXiv version; the verification
software and run log are supplied as ancillary files.

\smallskip
\noindent\textbf{Online Resource 1.} Detailed support-function certificates,
obstruction-cone calculations, and subdivision run tables.

\smallskip
\noindent\textbf{Online Resource 2.} Exact-arithmetic verification software,
execution instructions, and the complete verification log.

\end{document}


\maketitle
\begin{center}
\small
Independent, Sunnyvale, California, USA\\
Corresponding author: Parsa Bakhtary,
\href{mailto:pbakhtary@gmail.com}{pbakhtary@gmail.com}
\end{center}
This supplement records explicit certificates and subdivision counts.
Notation follows the article: $\Sig(a,b)$ is the eight-ray threefold fan,
$s_j=(-1,b-j,0)$ for $1\le j\le b$, and $\delta_\omega$ is the primitive
integral wall bend. Online Resource~2 contains the verification scripts.

\section{Closed-form support function for the smooth chain}

The smooth chain in the main text admits an explicit support function.
Select all walls with normal $m_y=(0,1,0)$, namely
$\cV=\{45,58,78\}$. On the original rays, set
\[
 h_0=(-(b+1),0,0,1,0,-1,-(b+3),-a(b+3)-1).
\]
It has positive bends at the other fifteen walls, zero bends at $45,78$,
and bend $-a(b+3)$ at $58$. Its pullback satisfies $h_0(s_j)=-(b+1)$.
Writing $|y|$ for the function $(x,y,z)\mapsto|y|$, put
\begin{equation}\label{eq:explicit-family-support}
\begin{gathered}
 T_j=\frac{j(2b+1-j)}2,\qquad g(v_i)=0,\qquad g(s_j)=T_j,\\
 C=a(b+3)+2,\qquad\varepsilon=\frac1{b(b+1)},\qquad
 H=h_0-C|y|+\varepsilon g.
\end{gathered}
\end{equation}
Here $g$ is obtained by linear interpolation on each cone of the adapted fan. Its bends at walls
inside original maximal cones all equal $1$. With $s_0=v_1$ and $T_0=0$, the interior
relations
\[
 s_{j-1}+s_{j+1}=2s_j\quad(1\le j<b),\qquad s_{b-1}+v_2=s_b,
\]
have bends $2T_j-T_{j-1}-T_{j+1}=1$ and $T_b-T_{b-1}=1$.
Thus $g$ is strictly upper convex relative to $\Sig(a,b)$. We next
check that every bend of $H$ is positive.

\begin{proposition}\label{prop:explicit-support}
The function $H$ in~\eqref{eq:explicit-family-support} has positive bends
at all $18+3b$ walls of the adapted fan. Its minimum bend is exactly
$1/[b(b+1)]$.
\end{proposition}
\begin{proof}
Table~\ref{tab:final-bends} follows from the unchanged wall relations,
the interior relations above, and the chain relations
\[
 v_4+v_7=(b-j+1)s_j-(b-j)s_{j-1}\quad(1\le j\le b)
\]
after substituting~\eqref{eq:explicit-family-support}.
For the row $\langle s_{j-1},s_j\rangle$, use the identity
\[
 2T_j+(b-j)(b+1-j)=b(b+1)=\varepsilon^{-1},
\]
which converts the directly computed bend
$1+\varepsilon[T_j+(b-j)(b+1-j)]$ into $2-\varepsilon T_j$.
Since $T_j\le b(b+1)/2$, the chain bends are at least $3/2$.
Also $1-b\varepsilon=b/(b+1)\ge\varepsilon$. Every other entry is at least
$\varepsilon$, with equality at intermediate interior walls for $b\ge2$
and at $14,17$ for $b=1$. This proves the claimed minimum.
\end{proof}

\begin{table}[htbp]
\centering
\caption{All bends of the explicit final support function $H$;
$s_0=v_1$, and $C,\varepsilon,T_j$ are as in~\eqref{eq:explicit-family-support}.}
\label{tab:final-bends}
\small
\begin{tabular}{@{}lc@{}}
\toprule
Wall(s)&Bend of $H$\\
\midrule
$14,17$&$1-b\varepsilon$\\
$16$&$b+1$\\
$23$&$b+2$\\
$24,27$&$b+\tfrac12$\\
$34,37$&$1$\\
$35,56$&$a(b+3)$\\
$38,68$&$b+3$\\
$45,78$&$2C$\\
$46,67$&$2$\\
$58$&$2C-a(b+3)$\\
$\langle v_4,s_j\rangle,\langle v_7,s_j\rangle$, $1\le j<b$&$\varepsilon$\\
$\langle v_4,s_b\rangle,\langle v_7,s_b\rangle$&$2C+\varepsilon$\\
$\langle s_{j-1},s_j\rangle$, $1\le j\le b$&$2-\varepsilon T_j$\\
$\langle s_b,v_2\rangle$&$\tfrac32$\\
\bottomrule
\end{tabular}
\end{table}

\section{Dimension of the diagonal obstruction cone}

A single wall can meet every obstruction even when the obstruction
cone has large dimension, measured by the dimension of its linear span.

\begin{proposition}\label{prop:obstruction-dimension}
For $\Sig_k=\Sig(k,k)$,
\[
 \dim\cC(\Sig_1)=2,\qquad \dim\cC(\Sig_k)=13\quad(k\ge2).
\]
\end{proposition}
\begin{proof}
The wall relations give two independent nonnegative dependencies:
\begin{align}
 \delta_{12}+\delta_{37}+\delta_{46}+\delta_{58}
 +(k-1)(\delta_{14}+\delta_{38})&=0,\label{eq:z13-obstruction}\\
 \delta_{12}+\delta_{34}+\delta_{58}+\delta_{67}
 +(k-1)(\delta_{14}+\delta_{38})&=0.\label{eq:z13-second-obstruction}
\end{align}
For $k=1$, the values $(-1,0,0,0,0,-1,-1,-1)$ have zero bends exactly at
$12,34,37,46,58,67$ and positive bends elsewhere. Every nonnegative
relation is supported on those six rows. They have rank four: the rows
$12,34,37,46$ are independent, while the displayed relations express the
remaining two in their span. The two independent nonnegative dependencies
therefore give an obstruction cone of dimension two.

For $k\ge2$, we first show that every support function with nonnegative
bends is globally linear. Let $h$ be such a function. Subtract a global
linear function to normalize $h(v_2)=h(v_3)=h(v_5)=0$, and write
\[
 A=h(v_1),\quad B=h(v_4),\quad U=h(v_6),\quad V=h(v_7),\quad D=h(v_8).
\]
Set $Y=\delta_{14}=-U\ge0$, $Z=\delta_{38}=-V\ge0$. The identities
\[
 \delta_{12}+\delta_{46}=-kY+Z,\qquad
 \delta_{37}+\delta_{58}=Y-kZ
\]
force $Z\ge kY$ and $Y\ge kZ$, hence $Y=Z=0$. Both sums vanish and their
four nonnegative summands vanish individually. Their formulas then give
$U=V=D=0$ and $A=B$. Finally $\delta_{34}=B\ge0$ and
$\delta_{24}=-A\ge0$ force $A=B=0$.

On the five-dimensional space of ray values modulo global linear
functions, the cone generated by the eighteen bend rows has dual cone
$\{0\}$, so it is the entire dual space. The negative of each row is
therefore a nonnegative combination of rows. This gives a dependence
with positive coefficient at the chosen row. Summing these dependences
gives a kernel vector with every coordinate positive.
A small neighborhood of this vector in the linear kernel remains
positive, so the obstruction cone has the full kernel dimension.
The row rank is five, since zero bends characterize globally linear
functions. Thus the dimension is $18-5=13$.
\end{proof}

\section{Basis-ordered full runs and a singular comparison}
The full runs use the standard basis and the normal order and maximal
sum-weight rule of the main text. Ties use the lexicographically smallest
sorted pair of primitive ray vectors. Table~\ref{tab:full-runs} lists the
per-normal counts. Selective coordinate adaptation costs one step for
Oda's fan and two for $\Sig(2,2)$.
\begin{table}[H]
\centering
\caption{Subdivisions per original wall normal in the full sign-adaptation runs}
\label{tab:full-runs}
\small
\begin{tabular}{@{}cr@{\qquad\qquad}cr@{}}
\toprule
\multicolumn{2}{c}{Oda's fan}&\multicolumn{2}{c}{$\Sig(2,2)$}\\
Normal&Steps&Normal&Steps\\
\midrule
$(0,0,1)$&1&$(0,0,1)$&2\\
$(0,1,-1)$&3&$(0,1,0)$&2\\
$(0,1,0)$&1&$(1,0,-1)$&2\\
$(1,-1,-1)$&5&$(1,0,0)$&1\\
$(1,-1,0)$&2&$(1,1,-1)$&3\\
$(1,-1,1)$&5&$(1,1,0)$&2\\
$(1,0,-1)$&2&$(2,0,-1)$&4\\
$(1,0,0)$&1&$(2,1,-2)$&9\\
$(1,1,-1)$&5&$(2,1,0)$&9\\
&&$(2,2,-1)$&8\\
\midrule
Total&25&Total&42\\
\bottomrule
\end{tabular}
\end{table}

For the singular comparison, replace each Oda ray $v_i=(x_i,y_i,z_i)$
by $v'_i=\prim(x_i,2y_i,z_i)$ in $\ZZ^3$, retaining the cone list.
There are initially four smooth and six multiplicity-two maximal cones.
The rational linear change preserves completeness and nonprojectivity,
and preserves the obstruction transversal $\{25\}$, whose normal is
the character $(x,y,z)\mapsto z$.
The only bad edge is $\langle v'_3,v'_6\rangle$, and its primitive sum is
\[
 s=\tfrac12((0,0,1)+(0,-2,-1))=(0,-1,0).
\]
The two affected multiplicity-two cones become four smooth cones.
The final distribution is eight smooth and four multiplicity-two cones.
The values $(0,0,-12,-42,-27,-36,-33,-13)$, in the order
$(v'_1,\ldots,v'_7,s)$, have minimum primitive integral bend one,
certifying projectivity.

\section{Runs of the direct labelled construction}
Table~\ref{tab:labelled-runs} gives the sizes of the barycentric
arrangement and its labelled resolution. These runs use the direct fan
construction, not the principalization algorithm for smooth input.
\begin{table}[H]
\centering
\caption{Barycentric arrangement and final resolution sizes}
\label{tab:labelled-runs}
\small
\begin{tabular}{@{}lrrrr@{}}
\toprule
Input&Initial rays&Maximum multiplicity&Final rays&Final maximal cones\\
\midrule
Oda&146&1&146&288\\
$\Sig(1,1)$&74&1&74&144\\
$\Sig(2,2)$&210&3&274&544\\
\bottomrule
\end{tabular}
\end{table}
For $\Sig(2,2)$, the successive label tuples $(1,3)$, $(1,2,3)$, and
$(1,5)$ select $16$, $32$, and $16$ cones. The script
\texttt{verify\_certificates.py} checks strict rational support functions
after every batch. It repeats each run after a cyclic coordinate
permutation, a unimodular shear, and a lattice change of determinant $-1$,
reversing the cone order each time. All nine outputs and support functions
agree with their transported counterparts.